\documentclass[a4paper,reqno,12pt]{amsart}
\usepackage[all]{xy}
\usepackage{eucal}

\usepackage[T1]{fontenc}
\usepackage{textcomp}
\usepackage{amsmath,amsthm,amssymb,amscd}
\usepackage{type1cm}
\usepackage{txfonts}

\numberwithin{equation}{section}	

\theoremstyle{plain}
\newtheorem{thm}{Theorem}[section]      
\newtheorem*{thm*}{Theorem}

\newtheorem{lem}[thm]{Lemma}

\newtheorem{prop}[thm]{Proposition}

\newtheorem{cor}[thm]{Corollary}

\theoremstyle{definition}
\newtheorem{definition}[thm]{Definition}

\newtheorem{example}[thm]{Example}

\theoremstyle{remark}
\newtheorem{remark}[thm]{Remark}

\theoremstyle{remark}
\newtheorem{remarks}[thm]{Remarks}

\def\geqsl{\geqslant}
\def\leqsl{\leqslant}

\newcommand{\tp}[1]{\ensuremath{\operatorname{}^{t}\hspace{-3truept}{#1}}}

\newcommand{\Mat}[1]{\ensuremath{\operatorname{Mat}_{#1}}}
   
\newcommand{\Lie}[1]{\ensuremath{\operatorname{Lie}{#1}}}

\newcommand{\Ad}{\ensuremath{\operatorname{Ad}}}  
\newcommand{\ad}{\ensuremath{\operatorname{ad}}}  

\renewcommand{\Omega}{\ensuremath{\varOmega}}
\renewcommand{\Xi}{\ensuremath{\varXi}}
\renewcommand{\Theta}{\ensuremath{\varTheta}}

\newcommand{\R}{\ensuremath{\mathbb R}}
\newcommand{\C}{\ensuremath{\mathbb C}}

\newcommand{\gl}{\ensuremath{\mathfrak {gl}}}

\newcommand{\g}{\ensuremath{\mathfrak{g}}}   
\newcommand{\q}{\ensuremath{\mathfrak{q}}}   
\newcommand{\h}{\ensuremath{\mathfrak{h}}}   

\renewcommand{\u}{\ensuremath{\mathfrak{u}}}

\newcommand{\GL}[1]{\ensuremath{\mathrm{GL}_{#1}}}
\newcommand{\SL}[1]{\ensuremath{\mathrm{SL}_{#1}}}
   
\newcommand{\U}{\ensuremath{\mathrm{U}}}   
\newcommand{\SU}{\ensuremath{\mathrm{SU}}}   
   
\newcommand{\Gc}{\ensuremath{G_{\C}}}
\newcommand{\Kc}{\ensuremath{K_{\C}}}

\newcommand{\dd}{\ensuremath{\operatorname{d}\!}}

\def\<{\langle}
\def\>{\rangle}

\def\c2vec#1#2{ %
   \left[ \begin{smallmatrix} %
           #1 \\ #2  \end{smallmatrix} %
   \right]}

\newcommand{\twdctgbdl}[1]{\ensuremath{T^*(#1)_{\lambda}}} 
\renewcommand{\Gc}{\ensuremath{G}}
\renewcommand{\Kc}{\ensuremath{K}}

\newcommand{\Gr}{\ensuremath{G_{\R}}}
\newcommand{\Kr}{\ensuremath{K_{\R}}}

\newcommand{\Qmax}{\ensuremath{Q}}  
\newcommand{\qmax}{\ensuremath{\mathfrak q}}  

\newcommand{\flag}{\ensuremath{\Gc/Q}}  
\newcommand{\lpsi}{\ensuremath{\psi_{\lambda;e}}} 
\newcommand{\lPsi}{\ensuremath{\Psi_{\lambda;e}}} 
\newcommand{\inv}[1]{\ensuremath{{#1\!\!\!}^{-1}}}

\title%
{A twisted moment map and its equivariance} 

\author{Takashi Hashimoto}
\thanks{Partly supported by the Grant-in-Aid for Scientific Research (C), Japan Society for the Promotion of Science.}

\address{
  University Education Center, 
  Tottori University, 
  4-101, Koyama-Minami, Tottori, 680-8550, Japan
    }
\email{thashi@uec.tottori-u.ac.jp}
\date{\today}

\begin{document}

\begin{abstract}
Let $G$ be a linear connected complex reductive Lie group. 
The purpose of this paper is to construct a $G$-equivariant symplectomorphism 
in terms of local coordinates
from a holomorphic twisted cotangent bundle of the generalized flag variety of $G$
onto the semisimple coadjoint orbit of $G$.
As an application,
one can obtain an explicit embedding of a noncompact real coadjoint orbit into the twisted cotangent bundle.
\end{abstract}

\maketitle

\section{Introduction}
\label{intro}

The main purpose of this paper is to construct an equivariant symplectomorphism 
concretely in terms of local coordinates
from a holomorphic twisted cotangent bundle of the complex generalized flag variety
onto the complex coadjoint orbit of a semisimple element.
As an application,
one can obtain an explicit embedding of a noncompact real coadjoint orbit into the twisted cotangent bundle.

More precisely,
let $\Gc$ denote a linear connected complex reductive Lie group with Lie algebra $\g$.
We fix a Cartan subalgebra $\h$ of $\g$, 
and consider a nonzero element $\lambda$ of $\h^{*}$, the dual space of $\h$.
Under the assumption that 
the isotropy subalgebra of $\lambda$ in $\g$ is distinct from $\g$,
take a parabolic subgroup $Q$ of $\Gc$
whose Levi factor is the isotropy subgroup of $\lambda$ in $\Gc$,
and let $\{ U_{\sigma} \}_{\sigma \in W/W_\lambda}$ 
be the open covering of the flag variety $\Gc/Q$ indexed by $W/W_{\lambda}$
(see \eqref{e:open_covering_of_the_flag_variety} below for details).
Then, based on the key observation 
that cotangent vectors can be written in terms of one-form taking values in a subalgebra of $\g$
(Lemma \ref{l:key_relation} below), 
we construct a holomorphic isomorphism $\mu_{\lambda;\sigma}$
from the cotangent bundle $T^* U_\sigma$
into the complex coadjoint orbit $\Omega_{\lambda} := \Gc \cdot \lambda$ for each $\sigma$.
Note that $U_\sigma$ is homeomorphic to $\C^n$ and that $T^{*} U_\sigma$ is trivial, 
i.e., $T^{*} U_\sigma \simeq U_\sigma \times \C^n$ with $n=\dim (\flag)$ for each $\sigma$.
We shall see that the isomorphisms $\{ \mu_{\lambda;\sigma} \}_{\sigma \in W/W_\lambda}$ 
are closely related to the triangular decomposition of $\Gc$ (or rather, of $\g$).

A prototype of the isomorphism $\mu_{\lambda;\sigma}$ was obtained 
in the process of proving the formula for the generating function 
of the principal symbols of the invariant differential operators
that play an essential role in the Capelli identities 
in the case of Hermitian symmetric spaces
(see \cite{HU91} and \cite{KinvDO}).
Namely, let $(\Gr,\Kr)$ be a classical Hermitian symmetric pair 
of noncompact type
such that $\Gr$ is a real form of $\Gc$,
and assume that $\lambda$ is proportional to 
the fundamental weight corresponding to the unique noncompact simple root.
Then the symbols of the differential operators 
$\pi_{\lambda}(X)$, $X \in \g$,
naturally constitute a holomorphic isomorphism 
from $T^{*} (\Gr/\Kr)$ into $\g \simeq \g^{*}$,
where $\pi_{\lambda}$ denotes the representation 
of the complex Lie algebra $\g$ induced from that of $\Gr$,
the so-called holomorphic discrete series representation.
Note that the Hermitian symmetric case corresponds to 
the case where the flag variety $\flag$ is Grassmannian.

Furthermore, 
the isomorphisms $\{ \mu_{\lambda;\sigma}\}$ acquire equivariance under $\Gc$
if we let $\Gc$ act on $T^* U_\sigma$ by affine transformation 
instead of the canonical linear one.
Since each coset of $W/W_{\lambda}$ is represented by an element $\dot\sigma \in \Gc$,
we can glue together the trivial bundles $\{ T^* U_\sigma \}_\sigma$ 
by transition functions induced from the affine action of $\Gc$
to form a holomorphic twisted cotangent bundle
which we denote by $\twdctgbdl{\flag}$ in this paper.
Since the twisted cotangent bundle is locally isomorphic 
to the (standard) holomorphic cotangent bundle $T^*(\flag)$ by its construction,
we can define local isomorphisms 
from $T^*(\flag)_\lambda|_{U_\sigma}$ 
into $\Omega_\lambda$
by the same formulae as $\{\mu_{\lambda;\sigma}\}$,
which satisfy the compatibility condition
\[
 \mu_{\lambda;\sigma}|_{\varpi^{-1}(U_\sigma \cap U_\tau)}
  =  \mu_{\lambda;\tau}|_{\varpi^{-1}(U_\sigma \cap U_\tau)}
\]
for $\sigma$ and $\tau \in W/W_\lambda$,
where $\varpi:\twdctgbdl{\flag} \to \flag$ is the projection.
By patching together the maps $\{ \mu_{\lambda;\sigma}\}_\sigma$,
we obtain an holomorphic isomorphism $\mu_{\lambda}$ 
from the holomorphic twisted cotangent bundle 
onto the complex coadjoint orbit.

We remark that
when $\lambda=0$ the transition functions of the twisted cotangent bundle $\twdctgbdl{\flag}$,
which are affine transformations of the fibers as mentioned above,
reduces to the canonical transition functions of the cotangent bundle $T^*(\flag)$,
which are linear transformations of the fibers,
and that the map $\mu_{\lambda}$ reduces to the moment map 
from $T^* (\flag)$ into $\g^{*}$.
In this sense, 
the isomorphism $\mu_{\lambda}$ can be regarded as a twisted moment map 
(see \cite{Rossmann91} and \cite{SV96}).

It is well known that the coadjoint orbit possesses a canonical symplectic form called the Kostant-Souriau form, 
and it is shown that our twisted cotangent bundle possesses a (holomorphic) symplectic form, 
which will be denoted by $\omega$ below, expressed locally by the same formula 
as that of the canonical holomorphic symplectic form on the holomorphic cotangent bundle.
The isomorphism $\mu_{\lambda}$ preserves the symplectic forms.
Thus, 
the holomorphic map $\mu_\lambda$ gives a \Gc-equivariant symplectomorphism 
from the holomorphic twisted cotangent bundle of the complex generalized flag variety 
onto the complex coadjoint orbit.
Furthermore, it provides a moment map on the symplectic \Gc-manifold $(\twdctgbdl{\flag},\omega)$.

The rest of this paper is organized as follows.
In Section 2,
we first review the case of the Hermitian symmetric by a basic example, 
then construct holomorphic local isomorphisms $\mu_{\lambda;\sigma}$
from cotangent bundle $T^* U_\sigma$ into the coadjoint orbit 
for $\sigma \in W/W_\lambda$.
In Section 3,
we define an action of $\Gc$ on the cotangent bundle by affine transformation,
and show that $\mu_{\lambda;\sigma}$ is $\Gc$-equivariant.
Replacing the canonical transition functions 
by the ones induced from the affine action of $\Gc$,
we construct the holomorphic twisted cotangent bundle mentioned above,
and show that the maps $\{ \mu_{\lambda;\sigma} \}_\sigma$ 
provides the isomorphism $\mu_{\lambda}$ 
from the twisted cotangent bundle onto the complex coadjoint orbit.
Finally, we prove that the isomorphism $\mu_{\lambda}$ is symplectic.
As an application, we give an explicit embeddings of coadjoint \Gr-orbits into the twisted cotangent bundles 
when $(\Gr,\Kr)$ is a classical Hermitian symmetric space of noncompact type.


\section{Twisted moment map}
\label{sec:1}

Throughout,
let $\Gc$ be a linear connected complex reductive Lie group 
with Lie algebra $\g$.
We fix a Cartan subalgebra $\h$ of $\g$ 
and denote the dual space of $\h$ by $\h^{*}$.
Let $\g=\h \oplus \bigoplus_{\alpha \in \Delta} \g_{\alpha}$
be the root space decomposition
with $\Delta$ a root system of $\g$ with respect to $\h$.
Choosing a positive root system $\Delta^+ \subset \Delta$,
we set $\mathfrak b:=\h \oplus \bigoplus_{\alpha \in \Delta^+} \g_{-\alpha}$. 
We take a nonzero root vector $E_{\alpha}$ from $\g_{\alpha}$ 
for each $\alpha \in \Delta$.

\subsection{Hermitian symmetric space}
\label{sec1-1}

In this subsection,
let $(\Gr,\Kr)$ denote a classical Hermitian symmetric pair 
of noncompact type. 
Let $\Gc$ and $\Kc$ be the complexifications of $\Gr$ and $\Kr$ respectively
and $\Qmax$ a maximal parabolic subgroup of $\Gc$ 
whose Levi factor is $\Kc$.
Let $\mathfrak k$ and $\qmax$ be the Lie algebras of $K$ and $\Qmax$ 
respectively.
Denoting by $\u^{-}$ the nilradical of $\qmax$
and by $\u$ its opposite,
we put $U:=\exp \u$ and $U^{-}:=\exp \u^{-}$.

Consider a holomorphic character $\lambda : \Qmax \to \C^{\times}$
whose differential restricted on $\h$ is proportional to
the fundamental weight corresponding to the unique noncompact simple root.
Let us denote by $\C_{\lambda}$ 
the one-dimensional representation $(\lambda,\C)$ of $\Qmax$.
Then one can construct an irreducible unitary representation 
$(\pi_{\lambda},{\mathcal H}_{\lambda})$ of $\Gr$,
the so-called holomorphic discrete series representation,
by Borel-Weil theory as follows.
Let ${\mathcal L}_{\lambda}$ be the pull-back by the open embedding 
$\Gr/\Kr \hookrightarrow \Gc/\Qmax$ 
of the holomorphic line bundle $\Gc \times_{\Qmax} \C_{\lambda}$
associated to the principal bundle $\Gc \to \Gc/\Qmax$. 
The Hilbert space ${\mathcal H}_{\lambda}$ consists of 
square-integrable holomorphic sections for ${\mathcal L}_\lambda$,
which we identify with the space of holomorphic functions 
$f$ on the open subset $\Gr \Qmax \subset \Gc$ 
that satisfy the following conditions
\begin{equation*}
  f(xq) = \lambda(q)^{-1}f(x) \quad (x \in \Gr \Qmax, q \in \Qmax)
     \qquad  \textrm{and} \qquad
 \int_{\Gr} |f(g)|^2 \dd g < \infty,
\end{equation*}
where $\dd g$ denotes the Haar measure on $\Gr$.
If $\lambda$ satisfies certain conditions, $\mathcal H_\lambda$ is non trivial.
Now the irreducible unitary representation $\pi_{\lambda}$ of $\Gr$ is defined by
\[
 \pi_{\lambda}(g)f(x):=f(g^{-1} x)
     \quad  \textrm{for} \quad f \in {\mathcal H}_{\lambda}.
\]
This induces a complex linear representation of $\g$,
which we also denote by $\pi_{\lambda}$
(see \cite{Knapp86} for details).

Take a basis $\{X_i\}$ for $\g$,
and its dual basis $\{ X_i^{\vee} \}$,
i.e.~the basis for $\g$ satisfying that
\begin{equation*}
 B(X_i,X_j^{\vee})=\delta_{ij},
\end{equation*}
where $B$ is the nondegenerate invariant symmetric bilinear form on $\g$.
For $X \in \g$ given, 
denoting by $\sigma_{\lambda}(X)(x,v^*)$
the symbol of the differential operator 
$\pi_{\lambda}(X)$ at $x \in \Gr/\Kr$
with cotangent vector $v^* \in T^{*}_x(\Gr/\Kr)$,
we define
\[
 \mu_{\lambda;0}(x,v^*)
  :=\sum_{i} \sigma_{\lambda}(X_i)(x,v^*) \otimes X_i^\vee.
\]
Note that $\mu_{\lambda;0}$ is independent of the basis $\{ X_i \}$ chosen. 

Recall from \cite{MS95}, for example, 
that for a Lie group $A$ and an $A$-manifold $M$,
the cotangent bundle $T^{*}M$ 
is a symplectic $A$-manifold 
in the canonical way.
Namely, the Lie group $A$ acts on $T^{*}M$ by
\begin{equation}
\label{e:canonical_action_on_ctg_bdle}
 g.(x,v^*)=(g.x, (g^{-1})^{*} v^*)
\end{equation}
for $x \in M$ and $v^* \in T^{*}_x M$, 
where $(g^{-1})^{*}$ denotes the transpose map of the differential
$(g^{-1})_{*}: T_{g.x} M \to T_{x} M$
induced from the translation by 
$g^{-1}$ 
on the base manifold $M$.

Then, 
for $X \in \mathfrak{a}:=\Lie(A)$, 
the moment map defined on the cotangent bundle 
$\mu : T^{*} M \to \mathfrak{a}^{*}$ is given by
\begin{equation}
\label{e:def_of_moment_map}
 \langle \mu(x,v^*),X \rangle = v^*(X_{M}(x))
         \qquad  
      (x \in M, v^* \in T^{*}_{x} M), 
\end{equation}
where $\mathfrak{a}^{*}$ denotes the dual of $\mathfrak{a}$,
$\langle \cdot,\cdot \rangle$ the canonical pairing between $\mathfrak{a}^{*}$ and $\mathfrak{a}$,
and $X_{M}$ the vector field on $M$ generated by $X$:
\begin{equation}
\label{e:def_vec_field}
 X_{M}(x) \,\varphi = \left. \frac{\dd}{\dd t} \right|_{t=0} \varphi( \exp(-t X).x ) 
\end{equation}
for functions $\varphi$ defined around $x \in M$.

It follows from \eqref{e:def_of_moment_map} and 
\eqref{e:def_vec_field} that
the principal part of $\mu_{\lambda;0}$ is identical to 
the moment map $\mu: T^{*}(\Gr/\Kr) \to \g^{*}$
composed by the isomorphism $\g^{*} \simeq \g$ 
via the bilinear form $B$,
which we also denote by $\mu$.
Here $T^*(\Gr/\Kr)$ denotes the holomorphic cotangent bundle 
of the Hermitian symmetric space $\Gr/\Kr$.
Furthermore, 
the total symbol $\mu_{\lambda;0}$ 
can be regarded as a variant of the twisted moment map 
$\tilde\mu_{\lambda} : T^{*}(\Gr/\Kr) \to \g^{*} \simeq \g$
due to Rossmann (see \cite{Rossmann91}, or \S 7 of \cite{SV96}).
In fact, the difference $\tilde\mu_{\lambda}-\mu$,
which is denoted by $\lambda_{x}$ with $x \in \Gc/\Qmax$ therein,
can be expressed as 
$ \tilde\mu_{\lambda}-\mu = \Ad(g) \lambda^{\vee}$, or
\begin{equation}
\label{e:SV_mu_lambda}
 \tilde\mu_{\lambda}(x,v^*) = \Ad(g) \lambda^{\vee}+\mu(x,v^*), 
   \notag
\end{equation}
where 
$\lambda^{\vee} \in \g$ corresponds to $\lambda \in \g^{*}$ 
under the isomorphism $\g^{*} \simeq \g $ 
via the bilinear form $B$, 
and $g$ is an element of a compact real form $G_u$ of $\Gc$
such that $x=g.e_{\Qmax}$
with $e_{\Qmax}$ the origin of $\Gc/\Qmax$.
Now, if $x$ is in the open subset $\Gr/\Kr \subset \Gc/\Qmax$,
one can choose a unique element $u_x$ 
from a certain open subset of $U$ so that $x=u_x.e_{\Qmax}$,
instead of $g$ from $G_u$.
Then, one can immediately verify that
\begin{equation}
\label{e:our_mu_lambda}
\mu_{\lambda;0}(x,v^*)=\Ad(u_x)\lambda^{\vee}+\mu(x,v^*).
   \notag
\end{equation}
Moreover, the following relation holds:
\begin{align}
\label{e:equivariance_under_U}
 \Ad(u_x^{-1})\mu_{\lambda;0} (x,v^*)
  &= \mu_{\lambda;0} (\dot{e}, u_{x}^{*}\, v^*)
\end{align}
(see \cite{KinvDO}).

%
%
\begin{example}
\label{ex:su(p,q)}
Let $(\Gr,\Kr)=( \SU(p,q), \mathrm{S}(\U(p) \times \U(q) )$ $(p \geqsl q)$,
where we realize $\SU(p,q)$ as
\[
 \SU(p,q) =\{ g \in \SL{p+q}(\C); \tp{\bar{g}} I_{p,q} g=I_{p,q} \}  
\]
with
$ I_{p,q}
  =\left[
     \begin{smallmatrix} 1_p &  \\  & -1_q \end{smallmatrix}
   \right].
$
Then we take $K, \Qmax$ to be given by
\begin{align}
 K &=\left\{ \left[ \begin{smallmatrix} 
                        a & 0 \\ 0 & d 
		    \end{smallmatrix}
           \right] \in \SL{p+q}(\C)
                     ; a \in \GL{p}(\C), d \in \GL{q}(\C)
     \right\},
          \notag     \\
\Qmax &=\left\{ \left[ \begin{smallmatrix} 
                         a & 0 \\ c & d 
                      \end{smallmatrix}
               \right] \in \SL{p+q}(\C)
                     ; a \in \GL{p}(\C), d \in \GL{q}(\C), c \in \Mat{q,p}(\C) 
     \right\},
   \label{e:Qmax} 
\end{align}
respectively.
We can assume that
the holomorphic character $\lambda: \Qmax \to \C^{\times}$ in this case
is given by
\begin{equation}
\label{e:lambda_su(p,q)}
 \lambda (\begin{bmatrix} 
            a & 0 \\ c & d 
          \end{bmatrix}) 
  =(\det d)^{-s}
\end{equation}
for some integer $s$.

Note that $\Gr/\Kr$ is isomorphic to 
the bounded symmetric domain given by
\[
 \left\{ z=(z^{i j}) \in \Mat{p,q}(\C) ; 
          1_q - \tp{\bar z}z \textrm{ is positive definite}
 \right\}.
\]
Therefore, 
we can take holomorphic coordinates 
$(z^{i j},\xi_{i j})_{i=1,\dots,p; j=1,\dots,q}$ 
around an arbitrary point $(x,v^*)$
on the whole $T^{*}(\Gr/\Kr)$ such that
\begin{equation}
\label{e:w_and_xi_su(p,q)}
 v^* = \sum_{i,j} \xi_{ij} \dd z^{ij}.
   \notag
\end{equation}
Using the fact that $\u$ is abelian,
it is easy to show that 
the right-hand side of \eqref{e:equivariance_under_U} equals
\begin{equation}
\label{e:Ad_w}
 \begin{bmatrix}
    \frac{q}{p+q} s 1_p & 0 \\[2pt] -\xi & -\frac{p}{p+q} s 1_q 
 \end{bmatrix}
\end{equation}
(Theorem 4.9, \cite{KinvDO}),
where 
we denote the complex $q \times p$-matrix $\tp{(\xi_{ij})}$ by $\xi$.
Thus, if $s \ne 0$ and if we put
\begin{equation}
\label{e:prototype_key_relation}
 w:=- s^{-1} \xi 
  \quad \textrm{and} \quad
 u^{-}_w := \left[ 
              \begin{matrix} 
                1 & 0 \\ w & 1 
              \end{matrix} 
            \right],
\end{equation}then it is immediate to show that
\eqref{e:Ad_w} is equal to $\Ad(u^{-}_w) \lambda^{\vee}$.
Hence
\begin{equation}
 \mu_{\lambda;0}(x,v^*) = \Ad(u_x) \Ad(u^{-}_w) \lambda^{\vee}.
   \notag
\end{equation}
This yields an injective holomorphic map 
\(
 \mu_{\lambda;0}: T^{*} (\Gr/\Kr) \to \g,
\)
which is a prototype of our main object.

Observe that
there is no need to restrict the domain of $\mu_{\lambda;0}$ to $T^{*}(\Gr/\Kr)$.
Indeed, 
it naturally extends to the holomorphic cotangent bundle 
of the open subset $U \Qmax/\Qmax \subset \Gc/\Qmax$ 
if we do not take the real form $\Gr$ into account.
Furthermore, we can take an arbitrary $\lambda \in \h^{*}$;
we shall carry out this extended case in the next subsection. 
\end{example}

\subsection{Generalized flag variety}

Let $\Gc$ be a linear connected complex reductive Lie group 
with Lie algebra $\g$, as above.
Consider a nonzero $\lambda \in \h^{*}$,
which is not necessarily the same as in the previous subsection.
Put
\[
 \mathfrak l:=\g(\lambda)=\{ X \in \g ; \ad^{*}(X)\lambda=0 \};
\]
we assume that $\mathfrak l$ is distinct from $\g$ throughout.
Let $\q$ be a parabolic subalgebra of $\g$ containing $\mathfrak b$ 
whose Levi part is $\mathfrak l$.
We assume that $\q$ is not necessarily maximal.
Let $\u^{-}$ be the nilradical of $\q$,
and $\u$ the opposite of $\u^{-}$.
Our assumption on $\q$ implies that
the subalgebras $\u^{-}$ and $\u$ need not be abelian.
At any rate, 
we have the following decompositions:
\begin{equation}
\label{e:triangular_decomp}
 \q = \mathfrak l \oplus \u^{-},
   \quad \textrm{and} \quad 
 \g = \u \oplus \mathfrak l \oplus \u^{-}.
\end{equation}
Denote by $\Delta(\u)$ the subset of $\Delta^+$ such that
$\u=\bigoplus_{\alpha \in \Delta(\u)} \g_{\alpha}$
and $\u^{-}=\bigoplus_{\alpha \in \Delta(\u)} \g_{-\alpha}$.

Let $L:=G(\lambda)=\{ g \in \Gc; \Ad^{*}(g) \lambda=\lambda\}$,
the isotropy subgroup of $\lambda$ in $\Gc$. 
Denoting the analytic subgroup of $\u^{-}$ (resp. $\u$)
by $U^{-}$ (resp. $U$),
let us introduce holomorphic coordinates $z=(z^\alpha)_{\alpha \in \Delta(\u)}$ on $U$ 
and $w=(w_\alpha)_{\alpha \in \Delta(\u)}$ on $U^-$
by parametrizing elements $u \in U$ and $u^- \in U^-$ as
\[
 u=\exp \sum_{\alpha \in \Delta(\u)} z^\alpha E_\alpha 
   \quad \text{and} \quad
 u^-=\exp \sum_{\alpha \in \Delta(\u)} w_\alpha E_{-\alpha },
\]
which we denote by $u_z$ and $u^-_w$ respectively.

Put $Q=LU^{-}=U^{-}L$,
and let $T^{*}(\flag)$ denote the \emph{holomorphic} cotangent bundle 
of the flag variety $\flag$ 
and
\begin{align*}
 p   &: G \to \flag, 
     \\
 \pi &: T^{*}(\flag) \to \flag
\end{align*}
the canonical projections.
Fixing a representative $\dot\sigma \in \Gc$ 
of each $\sigma \in W/W_{\lambda}$ once and for all,
let us identify $\sigma$ with $\dot \sigma$,
where $W_{\lambda}$ denotes the isotropy subgroup 
of $\lambda$ in the Weyl group $W$.
Take the open covering $\{ U_{\sigma}\}$ of $\flag$:
\begin{equation}
\label{e:open_covering_of_the_flag_variety}
\flag= \bigcup_{\sigma \in W/W_\lambda} U_\sigma
   \quad \textrm{with} \quad
  U_\sigma := \sigma U Q/Q.
\end{equation}
Since any element $x$ of $U_\sigma$ is expressed as
\begin{equation}
 x=\sigma u. e_Q 
\end{equation}
for a unique $u \in U$,
one can introduce holomorphic local coordinates 
$z_{\sigma}=({z_\sigma}^{\alpha})_{\alpha \in \Delta(\u)}$ 
on $U_\sigma$ by
\begin{equation}
\label{e:parametrization_of_u}
 u  = \exp \sum_{\alpha \in \Delta(\u)} {z_\sigma}^{\alpha} E_{\alpha}.
\end{equation}
We denote the element $u$ in \eqref{e:parametrization_of_u}
by $u_{z_\sigma}$ in what follows.
Then, every cotangent vector $v^* \in T^{*}_x (G/Q)$
can be written as
\[
 v^{*}
   = \sum_{\alpha \in \Delta(\u)} \xi_{\sigma \alpha} \dd {z_{\sigma}}^{\alpha},
\]
which provides holomorphic coordinates $(z_{\sigma},\xi_{\sigma})$ 
on $\pi^{-1}(U_{\sigma})$ 
with $\xi_\sigma=(\xi_{\sigma \alpha})_{\alpha \in \Delta(\u)}$.
In other words, one obtains a local triviality
\begin{equation}
 \phi_{\sigma}: \pi^{-1}(U_{\sigma}) 
   \xrightarrow{\sim} U_{\sigma} \times \C^n,
  \quad (x,v^*) \mapsto (z_{\sigma},\xi_{\sigma}),
\end{equation}
with $n=\# \Delta(\u)=\dim(\flag)$ for each $\sigma \in W/W_\lambda$. 
In the sequel, however,
if $(z_{e},\xi_{e})$ is in $\pi^{-1}(U_e)$ 
i.e., if $\sigma$ happens to equal the identity element $e$,
we suppress the subscripts and just write $(z,\xi)$ 
for brevity.

%
%
\begin{remark}
\label{r:key_remark}
For $\sigma \in W/W_{\lambda}$ given,
we have a unique decomposition 
$p^{-1}(U_{\sigma}) = \sigma U Q = \sigma U U^{-} L$.
Namely, any $g \in p^{-1}(U_{\sigma})$ 
uniquely factorizes into a product 
\begin{equation}
\label{e:big_cell_decomp}
 g=\sigma u u^{-} t 
 \quad (u \in U, u^{-} \in U^{-}, t \in L),
\end{equation}
which plays a role throughout the paper, 
as we shall see.
\end{remark}

Now, let us fix an element $\sigma \in W/W_{\lambda}$,
and discuss inside the product bundle $T^{*} U_{\sigma}=\pi^{-1}(U_{\sigma})$
until the end of this section.

%
%
\begin{lem}
\label{l:key_relation}
For $(z_{\sigma},\xi_{\sigma}) \in T^{*} U_{\sigma}$ given,
there correspond unique 
$u_{z_\sigma} \in U$ and $u^{-}_{w_\sigma} \in U^{-}$ such that
\begin{equation}
\label{e:key_relation}
 z_{\sigma} = \sigma u_{z_\sigma}.{e_Q}
     \quad \textrm{and} \quad
 \xi_{\sigma} 
     = -\< \Ad^{*}(u^{-}_{w_\sigma}) \lambda, \inv{u_{z_\sigma}} \dd u_{z_\sigma} \>,
\end{equation}
where we identify%
     \footnote{We shall sometimes use this convention throughout the paper.}
$\xi_\sigma=(\xi_{\sigma \alpha})_{\alpha \in \Delta(\u)}$ with 
$\sum_{\alpha \in \Delta(\u)} \xi_{\sigma \alpha} \dd {z_\sigma}^\alpha$,
which we abbreviate $\xi_\sigma \dd z_\sigma$.
\end{lem}
\begin{proof}
It is trivial that such $u_{z_\sigma} \in U$ uniquely exists.
If we identify $\g^{*}$ with $\g$ 
via the nondegenerate invariant symmetric bilinear form $B$ on $\g$,
the second formula of \eqref{e:key_relation}
can be rewritten as
\[
 B( \Ad(u^{-}_{w_\sigma}) \lambda^{\vee}, \inv{u_{z_\sigma}} \dd u_{z_\sigma} ) 
   = -\xi_{\sigma},  
\]
where $\lambda^{\vee} \in \h$ corresponds to $\lambda \in \h^{*}$
under the identification.

Now, since $u_{z_\sigma}^{-1} \dd u_{z_\sigma}$ is a 1-form taking values in $\u$,
the nondegeneracy of $B|_{\u^{-} \times \u}$ implies that
there exists a unique $Y \in \u^{-}$ satisfying
\[
 B(Y,\inv{u_{z_\sigma}} \dd u_{z_\sigma}) = -\xi_{\sigma}.
\]
Thus it suffices to show that
there exists a unique $u^{-}_{w_\sigma} \in U^{-}$ such that
\begin{equation}
\label{e:pairing_between_u_and_ubar}
 \Ad(u^{-}_{w_\sigma}) \lambda^{\vee} = \lambda^{\vee} + Y
\end{equation}
since $\Ad(u^{-}_{w_\sigma}) \lambda^{\vee}$ is 
in $\mathfrak l \oplus \u^{-}$ 
with its $\mathfrak l$-component equal to $\lambda^{\vee}$.
 Parametrizing $u^{-}_{w_\sigma}$ as
\begin{equation}
\label{e:parametrization_of_u^{-}}
 u^{-}_{w_\sigma}=\exp \sum_{\alpha \in \Delta(\u)} w_{\sigma \alpha} \, E_{-\alpha},
\end{equation}
one can determine the coefficients $w_{\sigma \alpha}$ 
inductively from \eqref{e:pairing_between_u_and_ubar}
with respect to the height of $\alpha$.
This completes the proof.
\end{proof}

%
%
\begin{example}
Let us consider the case where $\Gc=\GL{3}(\C)$ 
and a regular semisimple 
$\lambda= \sum_{i=1}^{3} \lambda_i \epsilon_i \in \h^{*}$, 
i.e., with $\lambda_i \ne \lambda_{j}$ if $i \ne j$.
Then $\mathfrak l=\g(\lambda)$ is equal to $\h$, the Cartan subalgebra 
consisting of all diagonal matrices in $\g=\gl_{3}(\C)$,
$\q$ the Borel subalgebra $\mathfrak{b}$ of all lower triangular matrices in $\g$,
and $\u^{-}$ (resp. $\u$) the nilpotent subalgebra 
of all strictly lower (resp. upper) matrices in $\g$.

We restrict ourselves to the case where $\sigma=e$ 
since the other cases are similar.
For $(z,\xi) \in T^{*} U_{e}$  
with $z=(z^{i,j})_{1 \leqsl i < j \leqsl 3}$ 
and $\xi \dd z=\sum_{1 \leqsl i < j \leqsl 3} \xi_{i,j} \dd z^{i,j}$,
if one writes 
\[
 u_z
   = \begin{bmatrix} 
        1 & z^{1,2} & z^{1,3}+\frac12 z^{1,2} z^{2,3} \\[1.5pt]
          & 1      & z^{2,3} \\ 
          &        & 1  
     \end{bmatrix},
   \quad
 u^{-}_{w}
   = \begin{bmatrix} 
        1                             &        &   \\
        w_{1,2}                        & 1      &   \\ 
        w_{1,3}+\frac12 w_{1,2} w_{2,3}  & w_{2,3} & 1  
     \end{bmatrix}
\]
as in \eqref{e:parametrization_of_u} and \eqref{e:parametrization_of_u^{-}},
the second formula of \eqref{e:key_relation} is equivalent to
\begin{align*}
 \lambda_{1,2} w_{1,2} 
  + \frac12 z^{2,3} \left( \lambda_{1,3} w_{1,3} + \frac12 (\lambda_{1,2}-\lambda_{2,3}) w_{1,2} w_{2,3} \right)
          &= -\xi_{1,2},
     \\
 \lambda_{2,3} w_{2,3} 
  - \frac12 z^{1,2} \left( \lambda_{1,3} w_{1,3} + \frac12 (\lambda_{1,2}-\lambda_{2,3}) w_{1,2} w_{2,3} \right) 
          &= -\xi_{2,3},
     \\
 \lambda_{1,3} w_{1,3} 
    + \frac12 (\lambda_{1,2}-\lambda_{2,3}) w_{1,2} w_{2,3} &= -\xi_{1,3},
\end{align*}
from which it immediately follows that
\begin{align*}
 w_{1,2} &= \frac1{\lambda_{1,2}}
              \Bigl( -\xi_{1,2} + \frac12 \xi_{1,3}z^{2,3} \Bigr), 
     \\ 
 w_{2,3} &= \frac1{\lambda_{2,3}}
              \Bigl( -\xi_{2,3} - \frac12 \xi_{1,3} z^{1,2} \Bigr), 
     \\
 w_{1,3} &= \frac1{\lambda_{1,3}}
              \left( -\xi_{1,3} 
                     - \frac12 \frac{\lambda_{1,2}
                     -\lambda_{2,3}}{\lambda_{1,2} \lambda_{2,3}} 
                        \Bigl( -\xi_{1,2}+\frac12 \xi_{1,3}z^{2,3} \Bigr)
                        \Bigl( -\xi_{2,3}-\frac12 \xi_{1,3}z^{1,2} \Bigr) 
              \right),
\end{align*}
where we put $\lambda_{i,j}:=\lambda_{i}-\lambda_{j}$ for $i \ne j$.

We remark that one can verify that the relation 
\eqref{e:equivariance_under_U} holds 
if one constructs an irreducible representation of $\GL{3}(\C)$ 
that is induced from the character $\lambda :Q \to \C^{\times}$ 
by Borel-Weil theory as in the previous subsection.
\end{example}

Put 
\(
\Omega_{\lambda}:=\Gc \cdot \lambda%
  =\{ \Ad^{*}(g)\lambda \in \g^{*}; g \in \Gc \},
\)
the coadjoint orbit of $\lambda$ under the complex Lie group \Gc. 
It is canonically isomorphic to $G/L$,
and we denote by $p_\lambda$ the canonical surjection
\[
 p_\lambda : \Gc \to \Omega_{\lambda},
     \quad g \mapsto \Ad^{*}(g) \lambda.
\]
%
%
\begin{definition}
\label{d:local_twisted_mmap}
By Lemma \ref{l:key_relation} above, 
one can define a holomorphic map 
\begin{equation}
\label{e:definition_of_local_mu}
 \mu_{\lambda;\sigma}: T^{*} U_{\sigma} \to \Omega_{\lambda} 
     \quad \textrm{by} \quad 
 \mu_{\lambda;\sigma} (z_{\sigma},\xi_{\sigma}) 
        := \Ad^{*}(\sigma u_{z_\sigma} u^{-}_{w_\sigma}) \lambda,
\end{equation}
where $u_{z_\sigma} \in U$ and $u^{-}_{w_\sigma} \in U^{-}$ are the unique elements 
corresponding to $(z_{\sigma},\xi_{\sigma}) \in T^{*} U_{\sigma}$ 
determined by the relation \eqref{e:key_relation}.
Note in particular that 
$\mu_{\lambda;\sigma}$ is injective.
\end{definition}
%
%
%
\begin{remarks}
\label{r:expression_of_covector}
(i)\;
If $g \in \Gc$ satisfies that
\[
 \mu_{\lambda;\sigma}(z_{\sigma},\xi_{\sigma})=\Ad^{*}(g) \lambda
\]
for $(z_{\sigma},\xi_{\sigma}) \in T^{*} U_{\sigma}$,
then there exists an element $t \in L$ such that 
$g=\sigma u_{z_\sigma} u^{-}_{w_\sigma} t$.
The correspondence 
$(z_{\sigma},\xi_{\sigma}) \mapsto g= \sigma u_{z_\sigma} u^{-}_{w_\sigma} t$
can be regarded as a section for the fibration 
$p^{-1}(U_{\sigma}) \to T^{*} U_{\sigma}$:
\begin{equation}
\label{cd:local_section} 
\xymatrix@C20mm@R12mm{
       &  p^{-1}(U_{\sigma}) 
                 \ar@{->}[d]^{p_\lambda}    
           \\
  T^{*} U_{\sigma} \ar@{->}_{\mu_{\lambda;\sigma}}^{\sim} [r] \ar@{..>}^{g} [ur] 
       &  \mu_{\lambda;\sigma}(T^{*} U_{\sigma}). 
}
\notag
\end{equation}

Now, 
let us define a $\g$-valued 1-form $\theta$ on $\Gc$ by
\begin{equation}
\label{e:maurer-cartan}
\theta_g:= g^{-1} \dd g
     \quad (g \in \Gc).
\end{equation}
By abuse of notation,
we use the  same symbol $\theta_g$ 
to denote the pull-back 
of the 1-form given in \eqref{e:maurer-cartan}
by the local section $g: T^{*} U_{\sigma} \to p^{-1}(U_{\sigma})$. 
Then the second formula of \eqref{e:key_relation} can be written as
\begin{equation}
\label{e:xi_by_theta_a}
\begin{aligned}
 \xi_{\sigma} 
  &= - \< \lambda, \theta_{\sigma a} \>
       \\
  &= - \< \lambda, \theta_{a} \>,
\end{aligned}
\end{equation}
where we set $a:=u_{z_\sigma} u^{-}_{w_\sigma}$ for brevity.
In fact,
since
\[
 (\sigma a)^{-1} \dd \,(\sigma a)
 = a^{-1} \dd a
 = \Ad(u^{-}_{w_\sigma})^{-1} (\inv{u_{z_\sigma}} \dd u_{z_\sigma})
    + \inv{u^{-}_{w_\sigma}\,} \dd u^{-}_{w_\sigma},
\]
the relation $\eqref{e:xi_by_theta_a}$ follows from the fact that
the second term is a 1-form taking values in $\u^{-}$.

(ii)\;
If one restricts $\xi$ given in \eqref{e:key_relation} or \eqref{e:xi_by_theta_a}
to the \textit{smooth} cotangent bundle of $\Gr/\Kr$
then one can obtain the reproducing kernel of the irreducible unitary representation 
$(\pi_{\lambda},{\mathcal H}_{\lambda})$ of $\Gr$ 
when $(\Gr,\Kr)$ is a Hermitian symmetric pair 
(see \cite{HOOS93} for details).

\end{remarks}

\section{\Gc -equivariance}

\subsection{Local \Gc-action}

First let us consider such elements $g \in \Gc$ that map $U_e$ onto itself.

%
%
\begin{definition}
\label{d:affine_transf}
For $(z,\xi) \in T^{*} {U_{e}}$, 
let $u_z \in U$ and $u^{-}_w \in U^{-}$ be the unique elements
determined by \eqref{e:key_relation}.
If $g \in \Gc$ satisfies that $g.z \in U_{e}$,
or equivalently, that $g u_z \in p^{-1}(U_{e})$,
then by Remark \ref{r:key_remark}, 
one can write 
\begin{equation}
\label{e:gauss_decomp_gu}
 g u_z = u_{g.z} u^{-}_{g;z} t_{g;z} 
   \quad \text{with} \quad
  u_{g.z}\in U, u^-_{g;z} \in U^{-}, t_{g;z} \in L,
   \notag
\end{equation}
from which it follows that
\begin{align}
\label{e:gauss_decomp_guu}
 g u_z u^{-}_w 
   &= u_{g.z} u^{-}_{g;z} t_{g;z} \cdot u^{-}_w
      \notag \\
   &= u_{g.z} \cdot u^{-}_{g;z} (t_{g;z} u^{-}_w t_{g;z}^{-1}) \cdot t_{g;z}.
\end{align}
In particular, we see that
$g u_z u^{-}_w$ lies in $U U^{-} L$, 
and that its $U$- and $L$-components are identical 
to those of $g u_z$ respectively
since $t_{g;z} u^{-}_w t_{g;z}^{-1}$ is in the subgroup $U^{-}$.

In view of \eqref{e:key_relation} and \eqref{e:gauss_decomp_guu},
it is natural to define a cotangent vector $\lpsi(g) \xi$ by
\begin{equation} 
\label{e:affine_transf}%
 \lpsi(g) \xi
  = - \< \Ad^{*}(u^{-}_{g;z;w}) \lambda, \, u_{g.z}^{-1} \dd u_{g.z} \>, 
\end{equation} 
where we set $u^{-}_{g;z;w}:=u^{-}_{g;z} (t_{g;z} u^{-}_{w} t_{g;z}^{-1})$ 
for brevity.
If we put $a := u_z u^{-}_w$,
then, as we noted in Remark \ref{r:expression_of_covector},
the right-hand side of \eqref{e:affine_transf} can be written as 
\begin{align}
 \lpsi(g) \xi
     &= - \< \lambda, \theta_{g a t_{g;z}^{-1}} \> 
    \label{e:affine_transf2} \\
     &= - \< \lambda, \theta_{ga} \> 
        + \< \lambda, \dd t_{g;z} t_{g;z}^{-1} \>
    \notag  \\
     &= (g^{-1})^{*} \xi
        + \< \lambda, \dd t_{g;z} t_{g;z}^{-1} \>
    \label{e:affine_transf3}
\end{align}
since $\theta_{ga}=(g^{-1})^{*} \theta_a$,
where $(g^{-1})^{*}$ denotes the transpose map of the differential 
$(g^{-1})_{*}:T_{g.z} U_{e} \to T_{z} U_{e}$
induced from the translation by $g^{-1} \in \Gc$ on $U_{e}$.

Note that \eqref{e:affine_transf} implies that 
$\lpsi(g) \xi$ belongs to $T^{*}_{g.z}(\flag)$
since the decomposition \eqref{e:gauss_decomp_guu} is holomorphic,
and that \eqref{e:affine_transf3} reduces to
the canonical $\Gc$-action on the cotangent bundle 
given by \eqref{e:canonical_action_on_ctg_bdle} when $\lambda=0$.
Furthermore,
it follows that the second term in \eqref{e:affine_transf3} is an exact 1-form
since $t_{g;z}$ is an element of $G(\lambda)$
\end{definition}

%
%
\begin{prop}
\label{p:pseudo-action}
Let $(z,\xi) \in T^{*} {U_{e}}$.
For $g, h \in \Gc$ 
such that both $h.z$ and $g h.z$ are in $U_{e}$,
we have
\begin{equation}
\label{e:pseudo-action}
\lpsi(g) (\lpsi(h) \xi) = \lpsi(g h) \xi.
\end{equation}
\end{prop}
\begin{proof}
We use \eqref{e:affine_transf2} to prove the proposition.
Let $u_z \in U$ and $u^{-}_w \in U^{-}$ be 
the unique elements determined by \eqref{e:key_relation}
and let $a=u_z u^{-}_w$.
Since both $h u_z$ and $g u_{h.z}$ are in $p^{-1}(U_{e})$ by assumption,
they decompose as 
\begin{align}
 h u_{z} &= u_{h.z} u^{-}_{h; z} t_{h;z} \in U U^{-}L, 
   \label{e:gauss_decomp_hu}
     \\
 g u_{h.z} &= u_{g h.z} u^{-}_{g;h z} t_{g;h z} \in U U^{-}L.
   \label{e:gauss_decomp_gu_hz}
\end{align}
Then, we see that
\begin{align}
\label{e:gauss_decomp_ghu}
 g(h u_z) &= g ( u_{h.z} u^{-}_{h;z} t_{h;z} )
     \notag \\
   &= u_{g h.z} u^{-}_{g;h.z} t_{g; h.z} 
         \cdot u^{-}_{h;z} t_{h;z}
     \notag \\
   &= u_{g h.z}  
         \cdot u^{-}_{g;h.z} ({t_{g;h.z}}u^{-}_{h;z} t_{g;h.z}^{-1}) 
         \cdot t_{g;h.z} t_{h;z}.
\end{align} 
Namely,
the $L$-component of $ g(h u_z)$ equals 
$t_{g;h.z} t_{h;z}$.

Now, it follows from \eqref{e:gauss_decomp_hu} and \eqref{e:gauss_decomp_gu_hz} 
that
\begin{align*}
\lpsi(g)(\lpsi(h) \xi) 
   &= \lpsi(g)( \lpsi(h) \< -\lambda, \theta_{a} \>  )
     \\
   &= \lpsi(g) \< -\lambda, \theta_{h a t_{h;z}^{-1}} \> 
     \\
   &= \< -\lambda, \theta_{g h a t_{h;z}^{-1} t_{g;h.z}^{-1}} \>.
\end{align*}
On the other hand, 
it follows from \eqref{e:gauss_decomp_ghu} that 
\begin{align*}
 \lpsi(g h) \xi
   &= \< -\lambda, \theta_{(g h) a (t_{g;h.z} t_{h,z})^{-1}}  \>
     \\
   &= \< -\lambda, \theta_{g h a t_{h;z}^{-1} t_{g;h.z}^{-1}} \>.
\end{align*}
This completes the proof.
\end{proof}

%
%
For $(z,\xi) \in T^{*} U_{e}$
and $g \in \Gc$ such that $g.z \in U_{e}$,
we define  
\begin{equation}
\lPsi(g): T^{*} {U_{e}} \to T^{*} {U_{e}}
     \quad \text{by} \quad
  \lPsi(g)(z,\xi) := (g.z, \lpsi(g) \xi).
\end{equation}  
Note that $\lPsi(g)|_{T^{*}_{z} (\flag)}=\lpsi(g)$ is a bi-holomorphic map 
from $T^{*}_{z}(\flag)$ onto $T^{*}_{g.z}(\flag)$ 
for all $z \in U_{e}$.

%
%
\begin{prop}
\label{p:Gc-equivariance} 
For $(z,\xi) \in T^{*}{U_{e}}$ and $g \in \Gc$
such that $g.z \in U_{e}$,
we have
\begin{equation}
\label{e:equivariant}
 \mu_{\lambda;e}( \lPsi(g) (z,\xi) )
  = \Ad^{*}(g) \mu_{\lambda;e} (z,\xi)
\end{equation}
\end{prop}
\begin{proof}
This is equivalent to the definition of $\lpsi(g)$, 
with $g \in \Gc$, as we shall see soon.
In fact, 
the elements of $U$ and $U^{-}$ 
corresponding to $\lpsi(g)\xi$ by $\mu_{\lambda;e}$
are $u_{g.z}$ and $u^{-}_{g;z;w}$ respectively
in the notation of Definition \ref{d:affine_transf}.
Therefore, the left-hand side of \eqref{e:equivariant} equals 
\begin{align*}
   & \Ad^{*}(u_{g.z} u^{-}_{g;z;w}) \lambda
     \\
 = & \Ad^{*}(g u_z u^{-}_w) \lambda
 =   \Ad^{*}(g) \Ad^{*}(u_z u^{-}_w) \lambda
     \\
 = & \Ad^{*}(g) \mu_{\lambda;e}(z,\xi).
\end{align*}
by \eqref{e:gauss_decomp_guu} and \eqref{e:definition_of_local_mu}.
\end{proof}

%
%
\begin{example}%
[Example \ref{ex:su(p,q)} continued]
\label{ex:su(p,q)_2}
Let $\Gc=\SL{p+q}(\C), \Qmax$ and $\lambda$ be as in Example \ref{ex:su(p,q)}.
Then, for $(z,\xi) \in T^{*} U_{e}$, 
let $g=\left[ \begin{smallmatrix} a & b \\ c & d \end{smallmatrix} \right]$
be an element of $\Gc$ such that $g.z \in U_{e}$.
If one writes 
\(
 \left[\begin{smallmatrix} 1 & \hat z \\ 0 & 1 \end{smallmatrix} \right]:=u_{g.z},
 \left[\begin{smallmatrix} 1 & 0 \\ \hat w & 1 \end{smallmatrix}\right]:=u^{-}_{g;z;w}
\)
and
\(
 \left[\begin{smallmatrix} \hat a & 0 \\ 0 & \hat d \end{smallmatrix}\right]:=t_{g;z}
\)
in the decomposition \eqref{e:gauss_decomp_guu},
an elementary matrix calculation shows that
\allowdisplaybreaks
\begin{align*}
\hat z &= (az+b)(cz+d)^{-1} = g.z, 
   \\
\hat w &=\bigl( c+(cz+d)w \bigr) \bigl( a-(az+b)(cz+d)^{-1}c \bigr)^{-1},
   \\
\hat a &= a-(az+b)(cz+d)^{-1}c,
   \\
\hat d &= cz+d.
\end{align*}
In particular, one sees
\begin{align*}
\dd \hat z
  &= \dd \, \bigl( (az+b)(cz+d)^{-1} \bigr)
     \\
  &= \bigl( a-(az+b)(cz+d)^{-1}c \bigr) \dd z (cz+d)^{-1}.
\end{align*}
Using the relation \eqref{e:key_relation}, i.e.,
$\hat \xi=-s \, \hat w$,
one has
\[
 \hat \xi \dd \hat z=(cz+d) \xi \dd z (cz+d)^{-1}-sc \dd z (cz+d)^{-1}.
\]
Taking the trace of the both sides,
one obtains \eqref{e:affine_transf3} in this case;
in particular, the second term of \eqref{e:affine_transf3} is given by
\begin{align*}
 \<\lambda, \dd t_{g;z} t_{g;z}^{-1} \>
   &= -s \operatorname{tr} \bigl(c \dd z (cz+d)^{-1} \bigr) 
     \\
   &= -s \dd \, \log \det (cz+d),
\end{align*}
and hence \eqref{e:pseudo-action} corresponds to the cocycle condition of the automorphy factor.
\end{example}

%
%

\subsection{Global construction}

Recall that the flag variety $\flag$ has the open covering $\{ U_\sigma\}_{\sigma \in W/W_\lambda}$,
and that each $\pi^{-1}(U_\sigma)$ is bi-holomorphic to $U_\sigma \times \C^n$
with $n=\dim (\flag)$.
If a point $x \in \flag$ is in $U_\sigma \cap U_\tau$,
then it can be written
\[
 x= \sigma u_{z_\sigma} .e_Q = \tau u_{z_\tau}. e_Q.
\]
Therefore, 
if we take into account \eqref{e:key_relation}, \eqref{e:xi_by_theta_a} and 
Proposition \ref{p:Gc-equivariance}, 
it is natural from the group-theoretic point of view
to glue together the product bundles $\{ \pi^{-1}(U_\sigma) \}$, 
using the transition functions given by $\{ \lpsi(\tau^{-1} \sigma) \}$,
as follows.

In the disjoint union 
$\bigsqcup_{\sigma} (U_\sigma \times \C^n)$,
let us say that two points
$(z_\sigma,\xi_\sigma) \in U_\sigma \times \C^n$ 
and $(z_\tau,\xi_\tau) \in U_\tau \times \C^n$
are equivalent to each other
if and only if
\begin{equation}
\label{e:gluing_ctgbdle}
 \tau u_{z_\tau}. e_Q = \sigma u_{z_\sigma}. e_Q
   \quad \text{and} \quad 
 \xi_\tau = \lpsi(\tau^{-1} \sigma) \xi_\sigma,
\end{equation}
in which case we write $(z_\sigma,\xi_\sigma) \sim (z_\tau,\xi_\tau)$.
Then we define our twisted cotangent bundle to be 
the quotient space by this equivalence relation:
\begin{equation}
 \twdctgbdl{\flag} 
  := \bigsqcup_{\sigma \in W/W_\lambda} (U_\sigma \times \C^n) \, / \sim.
\end{equation} 
We denote by $[z_\sigma,\xi_\sigma]$ 
the equivalence class of $(z_\sigma,\xi_\sigma) \in U_\sigma \times \C^n$
and by $\varpi$ the projection 
\[
  \twdctgbdl{\flag} \to \flag, 
   \quad 
  [z_\sigma,\xi_\sigma] \mapsto \sigma u_{z_\sigma}. e_Q.
\]
Note that our twisted cotangent bundle is identical to 
the (usual) cotangent bundle $T^{*}(\flag)$ set-theoretically:
\[
 \twdctgbdl{\flag} = \bigcup_{x \in \flag} T^*_x (\flag),
\]
and that a local triviality on $\varpi^{-1}(U_\sigma)$ is given by
\[
 \varpi^{-1}(U_\sigma) 
  \simeq  U_\sigma \times \C^n, 
     \quad
 [z_\sigma,\xi_\sigma] \mapsto (z_\sigma,\xi_\sigma)
\]
for each $\sigma$.
Thus, our twisted cotangent bundle is locally isomorphic 
to the cotangent bundle 
and its transition functions 
are given in terms of the affine transformations $\lpsi$ 
which reduce to the canonical transition functions of the cotangent bundle $\twdctgbdl{\flag}$
when $\lambda=0$.

%
%
\begin{remark}
\label{r:symplectic_form_on_twdctgbdl}
Since the second term of \eqref{e:affine_transf3} is exact,
one obtains that 
$\dd \,( \xi_\sigma \dd z_\sigma) = \dd \,( \xi_\tau \dd z_\tau)$
on $\varpi^{-1}(U_\sigma \cap U_\tau)$.
Therefore, 
our twisted cotangent bundle possesses a holomorphic symplectic form 
that is identical to the canonical one 
on the cotangent bundle $T^{*}(\flag)$,
which we shall denote by $\omega$. 
\end{remark}

%
%
\begin{definition}
For given $g \in \Gc$ and 
$[z_\sigma,\xi_\sigma] \in \varpi^{-1}(U_\sigma) \subset \twdctgbdl{\flag}$,
take any $\tau \in W/W_\lambda$ such that $g.z_\sigma \in U_\tau$.
Suppose that $\xi_\sigma$ is written as
$\xi_\sigma=\< -\lambda,\theta_{u_{z_\sigma} u^-_{w_\sigma}} \>$ 
with $u_{z_\sigma}$ and $u^-_{w_\sigma}$ being the unique elements 
of $U$ and $U^-$ determined by \eqref{e:xi_by_theta_a}.
Then we define a cotangent vector 
$\psi_\lambda(g) \xi_\sigma \in T^*_{g.z_\sigma}(\flag)$ by
\begin{equation}
\label{e:global_action}
 \psi_\lambda(g) \xi_\sigma 
  := \lpsi(\tau^{-1} g \sigma) \<-\lambda, \theta_{u_{z_\sigma} u^-_{w_\sigma}} \>
\end{equation}
and a holomorphic map $\Psi_\lambda(g)$ 
from $\twdctgbdl{\flag}$ to itself by
\begin{equation}
 \Psi_\lambda(g) [z_\sigma,\xi_\sigma] 
  := [g.z_\sigma, \psi_\lambda(g) \xi_\sigma]. 
\end{equation}
\end{definition}

%
%
\begin{lem}
The map $\Psi_\lambda$ is well defined, 
i.e., it is independent of the choices of $\sigma$ and $\tau$ 
in \eqref{e:global_action} above.
Furthermore, we have
\begin{equation}
\Psi_\lambda(g) \Psi_\lambda(h) = \Psi_\lambda(g h)
\end{equation}
for all $g,h \in \Gc$.
Namely, $\Gc$ acts on the twisted cotangent bundle $\twdctgbdl{\flag}$ 
through $\Psi_\lambda$.
\end{lem}
\begin{proof}
Suppose that 
$[z_\sigma,\xi_\sigma]=[z_{\hat \sigma},\xi_{\hat \sigma}]$
and take another $\hat \tau$ such that $g.z_{\hat \sigma} \in U_{\hat \tau}$.
Then, by definition, 
one has
\[
 \xi_{\hat \sigma} =\lpsi ({\hat \sigma}^{-1} \sigma) \xi_\sigma.
\]

It suffices to show that
\[
 \lpsi( \hat \tau^{-1} g \hat \sigma ) \xi_{\hat \sigma}
  =\lpsi( {\hat \tau}^{-1} \tau) \lpsi( \tau^{-1} g \sigma) \xi_\sigma.
\]
Now one sees
\begin{align*}
 \lpsi( \hat \tau^{-1} g \hat \sigma ) \xi_{\hat \sigma} 
  &= \lpsi({\hat \tau}^{-1} g \hat \sigma) \lpsi( {\hat \sigma}^{-1} \sigma) \xi_\sigma 
        \\
  &= \lpsi({\hat \tau}^{-1} g \sigma) \xi_\sigma
        \\
  &= \lpsi({\hat \tau}^{-1} \tau) \lpsi(\tau^{-1} g \sigma) \xi_\sigma
\end{align*}
by Proposition \ref{p:pseudo-action}.
The second assertion also follows from Proposition \ref{p:pseudo-action}.
\end{proof}

Since $\varpi^{-1}(U_\sigma)=\pi^{-1}(U_\sigma)$,
one can define $\mu_{\lambda;\sigma} : \varpi^{-1}(U_\sigma) \to \Omega_\lambda$
by the same formula as \eqref{e:definition_of_local_mu} 
for each $\sigma \in W/W_\lambda$:
\begin{equation}
\label{e:local_def_of_mu}
 \mu_{\lambda;\sigma}: \varpi^{-1}(U_\sigma) \to \Omega_\lambda,
     \quad 
 [z_\sigma,\xi_\sigma] \mapsto \Ad^*(\sigma u_{z_\sigma} u^-_{w_\sigma}) \lambda,
\end{equation}
where $u_{z_\sigma} \in U$ and $u^-_{w_\sigma} \in U^-$ are determined 
by $(z_\sigma,\xi_\sigma)$ as in \eqref{e:key_relation}.

%
%
%
\begin{prop}
\label{p:compatibility_of_local_mu}
The local isomorphisms $\{ \mu_{\lambda;\sigma} \}_{\sigma \in W/W_\lambda}$ 
satisfy the compatibility condition
\begin{equation}
 \mu_{\lambda;\sigma}|_{\varpi^{-1}(U_\sigma \cap U_\tau)}
=\mu_{\lambda;\tau}|_{\varpi^{-1}(U_\sigma \cap U_\tau)}
     \quad
   (\sigma, \tau \in W/W_\lambda).
\end{equation} 
Thus we can define a globally defined bi-holomorphic map
\begin{equation}
\label{e:definition_of_global_mu}
 \mu_\lambda : \twdctgbdl{\flag} \to \Omega_\lambda 
     \quad \text{by} \quad
  \mu_\lambda|_{\varpi^{-1}(U_\sigma)} := \mu_{\lambda;\sigma}.
\end{equation}
Furthermore this map is $\Gc$-equivariant, 
i,e, we have
\begin{equation}
 \mu_\lambda \circ \Psi_\lambda(g) = \Ad^*(g) \circ \mu_\lambda
\end{equation}
for all $g \in \Gc$. 
\end{prop}
\begin{proof}
Suppose that a point of $\varpi^{-1}(U_\sigma \cap U_\tau)$ is expressed in two ways:
\[
 [z_\sigma,\xi_\sigma] = [z_\tau,\xi_\tau] \in \varpi^{-1}(U_\sigma \cap U_\tau),
\]
where we regard $[z_\sigma,\xi_\sigma] \in \varpi^{-1}(U_\sigma)$ 
and $[z_\tau,\xi_\tau] \in \varpi^{-1}(U_\tau)$. 
Let $u_{z_\sigma}, u_{z_\tau} \in U$
and $u^-_{w_\sigma}, u^-_{w_\tau} \in U^-$ satisfy
\[
 \xi_\sigma=\<-\lambda,\theta_{u_\sigma u^-_\sigma}\>
   \quad \text{and} \quad
 \xi_\tau=\<-\lambda,\theta_{u_\tau u^-_\tau}\>
\]
as in \eqref{e:xi_by_theta_a}
(we shall abbreviate $u_\sigma:=u_{z_\sigma}, u^-_\sigma:=u^-_{w_\sigma}$ 
etc.~ until the end of the proof).
Then by the definition of the equivalence relation \eqref{e:gluing_ctgbdle},
we have
\[
 \sigma^{-1} \tau u_\tau u^-_\tau = u_\sigma u^-_\sigma t
\]
for some $t \in L=G(\lambda)$.
Therefore, we see that
\begin{align*}
 \mu_{\lambda;\tau}([z_\tau,\xi_\tau])
  &=\Ad^*(\tau u_\tau u^-_\tau) \lambda 
   =\Ad^*(\sigma u_\sigma u^-_\sigma t) \lambda
     \\
  &=\Ad^*(\sigma u_\sigma u^-_\sigma) \lambda
     \\
  &=\mu_{\lambda;\sigma}([z_\sigma,\xi_\sigma]).
\end{align*}

Next, 
for $g \in \Gc$ and $[z_\sigma,\xi_\sigma] \in \varpi^{-1}(U_\sigma)$,
take $\hat \sigma \in W/W_\lambda$ satisfying $g.z_\sigma \in U_{\hat \sigma}$.
Since ${\hat \sigma}^{-1} g \sigma u_\sigma u^-_\sigma$ is in $U U^- L$,
it decomposes, say,
\[
 {\hat \sigma}^{-1} g \sigma u_\sigma u^-_\sigma 
  = u_1 u^-_1 t_1 
     \quad \text{with} \quad
  u_1 \in U, u^-_1 \in U^-, t_1 \in L.
\]
Then 
\begin{align*}
 \psi_\lambda(g) \xi_\sigma 
  &= \lpsi({\hat \sigma}^{-1} g \sigma) \xi_\sigma
     \\
  &= \lpsi({\hat \sigma}^{-1} g \sigma) \<-\lambda,\theta_{u_\sigma u^-_\sigma}\> 
     \\
  &= \<-\lambda,\theta_{{\hat \sigma}^{-1} g \sigma u_\sigma u^-_\sigma t_1^{-1}} \>
\end{align*}
Therefore, we see that
\begin{align*}
  \mu_\lambda(\Psi_\lambda(g) [z_\sigma,\xi_\sigma])
 &=\mu_{\lambda;\hat \sigma} ([g.z_\sigma, \psi_\lambda(g)\xi_\sigma])
      \\
 &=\Ad^*(\hat \sigma ) 
     \Ad^*({\hat \sigma}^{-1} g \sigma u_\sigma u^-_\sigma t_1^{-1}) \lambda
      \\
 &=\Ad^*(g) \mu_{\lambda}([z_\sigma,\xi_\sigma]).
\end{align*}
This completes the proof.
\end{proof}

%
%
\begin{example}
\label{ex:sl2}
Let us consider the case where $p=q=1$ in Example \ref{ex:su(p,q)}, 
i.e., $\Gc=\SL{2}(\C)$, 
$Q=\left\{ 
     \left[
       \begin{smallmatrix}
         a & 0 \\ c & a^{-1}
       \end{smallmatrix}
      \right] \in \Gc 
    \right\}$,
the Borel subgroup of \Gc,
and 
\[
\lambda:Q \to \C^{\times},
   \quad
     \left[
       \begin{matrix}
         a & 0 \\ c & a^{-1}
       \end{matrix}
      \right] 
  \mapsto a^s.
\]
The flag variety $\Gc/Q$ is identified with the complex projective line
$\mathbb{CP}^1$.
Under this identification,
the open covering $\{ U_e,U_\sigma\}$
(with 
\( \sigma=\left[ \begin{smallmatrix} 0 & 1 \\ -1 & 0 \end{smallmatrix} \right] \)
) 
is given by
\begin{equation*}
 U_e
  =\left\{ (z:1) \in \mathbb{CP}^1 ; z \in \C \right\} \simeq \C,
   \quad
 U_\sigma
  =\left\{ (1:z_\sigma) \in \mathbb{CP}^1 ; z_\sigma \in \C \right\} \simeq \C.
\end{equation*}
For $[z,\xi] \in \varpi^{-1}(U_e)$ 
and $[z_\sigma,\xi_\sigma] \in \varpi^{-1}(U_\sigma)$,
let $u_{z}, u_{z_\sigma} \in U$ and $u^-_{w}, u^-_{w_\sigma} \in U^-$ 
satisfy $\xi=-s w, \; \xi_\sigma=-s w_\sigma$ as in \eqref{e:key_relation}.
If $[z_\sigma,\xi_\sigma]=[z,\xi]$, 
then one sees
\begin{equation}
\label{e:sl2_case}
 z_\sigma= -\frac1{z}, 
   \quad 
 w_\sigma=z^2 w + z
\end{equation} 
since $ \xi_\sigma =\lpsi(\sigma^{-1}) \xi =z^2 \xi -s z$.

Denoting the maps $\mu_{\lambda;e}$ and $\mu_{\lambda;\sigma}$ 
followed by the isomorphism $\g^* \simeq \g$ 
via the trace form by the same notations,
one obtains
\begin{equation}
\label{e:mu_z_xi}
 \mu_{\lambda;e}([z,\xi])
  = \frac{s}2
     \begin{bmatrix} 
      1+2z w & -2 z (1+z w) \\ 
       2 w     & -(1+2 z w) 
     \end{bmatrix}
\end{equation}
and
\[
 \mu_{\lambda;\sigma}([z_\sigma,\xi_\sigma])
  = \frac{s}2
     \begin{bmatrix} 
      -(1+2z_\sigma w_\sigma) & -2 w_\sigma \\ 
      2 z_\sigma (1+ z_\sigma w_\sigma) & 1+2 z_\sigma w_\sigma
     \end{bmatrix},
\]
which, under the relation \eqref{e:sl2_case}, coincide with each other.
\end{example}

\subsection{Symplectomorphism}

%
%
We next prove that 
the map $\mu_{\lambda}:\twdctgbdl{\flag} \to \Omega_{\lambda}$ is symplectic.
Let $\omega$ and $\omega_{\lambda}$ denote 
the canonical $\Gc$-invariant holomorphic symplectic forms 
on $\twdctgbdl{\flag}$ and $\Omega_{\lambda}$ respectively. 
Recall that $\omega$ is defined by 
\begin{equation}
 \omega_{[z_\sigma,\xi_\sigma]} 
   = - \dd \, (\xi_{\sigma} \dd z_{\sigma})
   =\sum_{\alpha \in \Delta(\u)} \dd {z_{\sigma}}^{\alpha} \wedge \dd \xi_{\sigma \alpha}     
\end{equation}
if $[z_\sigma,\xi_\sigma] \in \varpi^{-1}(U_{\sigma}) \subset \twdctgbdl{\flag}$
(cf.~ Remark \ref{r:symplectic_form_on_twdctgbdl}),
and that $\omega_{\lambda}$ is defined by
\begin{equation}
\label{e:definition of omega_lambda}
 (\omega_{\lambda})_f( X_{\Omega_{\lambda}}, Y_{\Omega_{\lambda}})
    = - \<f, [X,Y] \>
     \quad (f \in \Omega_{\lambda}; X,Y \in \g),
\end{equation}
where $X_{\Omega_{\lambda}}, Y_{\Omega_{\lambda}}$ 
are the vector fields on $\Omega_{\lambda}$
generated by $X,Y \in \g$ respectively
that are defined by \eqref{e:def_vec_field}.

%
%
\begin{prop}
Let $\omega$ and $\omega_{\lambda}$ be the canonical symplectic forms 
on $\twdctgbdl{\flag}$ and $\Omega_{\lambda}$ respectively.
Then $\mu_{\lambda}$ preserves the symplectic forms:
\begin{equation}
\label{e:mu_lambda_is_symplectic}
\mu_{\lambda}^{*} \omega_{\lambda} = \omega. 
\end{equation}
\end{prop}
\begin{proof}
It suffices to show the equality \eqref{e:mu_lambda_is_symplectic} 
on the dense open subset $\varpi^{-1}(U_{e})$.
For $[z,\xi] \in \varpi^{-1}(U_{e})$,
we put $g:=u_z u^{-}_w $,
where $u_z \in U$ and $u^{-}_w \in U^{-}$ are determined by \eqref{e:key_relation}.
Then it follows from \eqref{e:xi_by_theta_a} that
\begin{equation}
 \omega = \< \lambda, \dd \theta_g \> 
        = \< -\lambda, \theta_g \wedge \theta_g \>
\end{equation}
since 
\(
\dd \theta_g=\dd \, ( g^{-1} \dd g)=-g^{-1} \dd g \wedge g^{-1} \dd g%
 =- \theta_g \wedge \theta_g.
\)
Thus, if we can show that
\begin{equation}
\label{e:Omega_by_theta}
 (\omega_{\lambda})_{g.\lambda}(X_{\Omega_\lambda},Y_{\Omega_\lambda})
 = \< -\lambda, %
      (\theta_{g.\lambda} \wedge \theta_{g.\lambda})(X_{\Omega_\lambda}, Y_{\Omega_\lambda}) 
   \>
   \notag
\end{equation}
for all $X,Y \in \g$,
then we are done.
Here, we set $g.\lambda=\Ad^{*}(g)\lambda$ for brevity, 
and denote by $\theta_{g.\lambda}$ the pull-back of the 1-form $\theta$ by the local section
\[
 g: \Omega_\lambda|_{\, \mu_{\lambda}(\varpi^{-1}(U_{e}))} \to \Gc,
   \quad
  g.\lambda \mapsto g
\]
with $g=u_z u^-_w$.

Now, for $X \in \g$,
we see that
\begin{align*}
 X_{\Omega_\lambda}({g.\lambda}) 
 & = \left. \frac{\dd}{\dd t} \right|_{t=0} \exp(-t X). (g.\lambda)
        \\
 & = \left. \frac{\dd}{\dd t} \right|_{t=0} g \exp(-t \Ad(g^{-1})X). \lambda
        \\
 & = g_{*} (\Ad(g^{-1})X)_{\Omega_\lambda} ({\lambda}).
\end{align*}
Hence we obtain that
\begin{align*}
 \theta_{g.\lambda}(X_{\Omega_\lambda})
  & = \theta_{g.\lambda}( g_* (\Ad(g^{-1})X)_{\Omega_\lambda} )
    = (g^{*} \theta)_{\lambda} ( (\Ad(g^{-1})X)_{\Omega_\lambda} )
        \\
  & = \theta_{\lambda} ( (\Ad( g^{-1}) X )_{\Omega_\lambda} )
        \\
  & = -\Ad(g^{-1}) X
\end{align*}
since $g^{*}\theta=\theta$ 
and $\theta_{\lambda}(Z_{\Omega_\lambda}) = -Z$ 
for $Z \in \g$.
Therefore, we have
\begin{align*}
 \< -\lambda,%
     (\theta_{g.\lambda} \wedge \theta_{g.\lambda})(X_{\Omega_\lambda}, Y_{\Omega_\lambda}) \>
 &= \< -\lambda,%
         [ \theta_{g.\lambda}(X_{\Omega_\lambda}), \theta_{g.\lambda}(Y_{\Omega_\lambda}) ] \>
     \\
 &= \< -\lambda, [\Ad(g^{-1})X, \Ad(g^{-1})Y] \>
     \\
 &= \< -\Ad^{*}(g)\lambda, [X,Y] \>,
\end{align*}
which equals 
$(\omega_{\lambda})_{g.\lambda}(X_{\Omega_\lambda},Y_{\Omega_\lambda})$ 
by definition.
\end{proof}

%
%
\begin{example}
We let $\Gc=\SL{2}(\C),Q$ and $\lambda \in \g^*$ be as in Example \ref{ex:sl2},
and still identify $\g^*$ with $\g$ by the trace form. 
If we parametrize an element $f \in \Omega_\lambda \subset \g$ as
\( f=\left[ \begin{smallmatrix} a & b \\ c & -a \end{smallmatrix} \right]\)
with $a,b$ and $c \in \C$,
then it is easy to show that the canonical symplectic form 
$\omega_\lambda$ on $\Omega_\lambda$ is given by
\[
 \omega_\lambda
   = \frac{2}{s^2}\left( a \dd b \wedge \dd c - b \dd a \wedge \dd c + c \dd a \wedge \dd b
          \right).
\]
If $f=\Ad(g) \lambda^\vee$ 
with $g=u_z u^-_{w} t \in p^{-1}(U_e) = U U^- L$
and if we write $u_z$ and $u^-_w$ as
\[
 u_z=\begin{bmatrix} 1 & \, z \\ 0 & 1  \end{bmatrix}, \;
 u^-_w=\begin{bmatrix} 1 & 0 \\ \, w & 1  \end{bmatrix}
     \quad (z,w \in \C),
\]
then we find that
\begin{equation*}
 \omega_\lambda 
   = -s \dd z \wedge \dd w
   = \< -\lambda, \theta_g \wedge \theta_g \>.
\end{equation*}

Now, for $[z,\xi] \in \varpi^{-1}(U_e)$, 
if $f=\mu_\lambda([z,\xi])=\Ad(u_z u^-_w) \lambda^\vee$, 
i.e., $\xi$ and $w$ are related by $\xi= -s w$ as in Lemma \ref{l:key_relation} 
or \eqref{e:prototype_key_relation},
then it is immediate to see that $\mu_\lambda$ preserves the symplectic forms.
\end{example}

%
%
Summarising the results, we have seen that the following holds (cf.~\cite{CG97, Graham_Vogan97}):
\begin{thm}
\label{t:equivariance}
The holomorphic map $\mu_{\lambda}: \twdctgbdl{\flag} \to \Omega_{\lambda}$ 
given by \eqref{e:local_def_of_mu} and \eqref{e:definition_of_global_mu}
is a \Gc-equivariant symplectic isomorphism.
\end{thm}

%
%
Furthermore, 
the map $\mu_\lambda$ provides a moment map $\twdctgbdl{\flag} \to \g^*$ with respect to the symplectic form $\omega$.
Namely, we have the following.
\begin{cor}
The action $\Psi_\lambda$ of $\Gc$ on the symplectic manifold $(\twdctgbdl{\flag},\omega)$ 
is Hamiltonian with moment map $\mu_\lambda$.
\end{cor}
\begin{proof}
Setting $M:=\twdctgbdl{\flag}$ for brevity, we must show that
\begin{equation*}
 \dd\, \< \mu_\lambda, X \> =\iota_{X_M} \omega
\end{equation*}
for $X \in \g$ with $\iota$ denoting the contraction.
Since $\{ Y_M; Y \in \g \}$ spans the holomorphic tangent space $T_\alpha M$,
it suffices to show that
\begin{equation}
\label{e:mu_lambda is a moment map}
 \dd\, \< \mu_\lambda(\alpha),X \> (Y_M) = \left( \iota_{X_M} \omega \right)_\alpha (Y_M) 
\end{equation}
for $X, Y \in \g$,
where we simply denote a point in $M$ by $\alpha$.

Now, since $\mu_\lambda$ is $\Gc$-equivariant, one has
\begin{align*}
\dd\, \< \mu_\lambda(\alpha),X \> (Y_M) 
  &= \left. \frac{\dd}{\dd t} \right|_{t=0} \< \mu_\lambda( \Psi_\lambda(\exp(-t Y)) \alpha ), X \>
   = \left. \frac{\dd}{\dd t} \right|_{t=0} \< \Ad^*(\exp(-t Y)) \mu_\lambda(\alpha), X \>
     \\
  &= \left. \frac{\dd}{\dd t} \right|_{t=0} \< \mu_\lambda(\alpha), \Ad(\exp t Y) X \>
   = \< \mu_\lambda(\alpha), [Y,X] \>
     \\
  &= - \< \mu_\lambda(\alpha), [X,Y] \>.
\end{align*} 
On the other hand, 
since $\mu_\lambda$ is symplectic, the right-hand side of \eqref{e:mu_lambda is a moment map} equals
\begin{equation*}
 \omega_\alpha (X_M, Y_M)
  = \left( \mu_\lambda^* \omega_\lambda \right)_\alpha (X_M, Y_M)
  = ({\omega_\lambda})_{\mu_\lambda(\alpha)} ( {\mu_\lambda}_* X_M, {\mu_\lambda}_* Y_M).
\end{equation*}
Here, using the \Gc-equivariance of $\mu_\lambda$ again, one has
\begin{align*}
  ({{\mu_\lambda}_*} X_M)({\mu_\lambda(\alpha)})
 &= \left. \frac{\dd}{\dd t} \right|_{t=0} \mu_\lambda( \Psi_\lambda(\exp(-t X)) \alpha)
  = \left. \frac{\dd}{\dd t} \right|_{t=0}  \Ad^*(\exp(-t X))\mu_\lambda( \alpha)
     \\
 &= X_{\Omega_\lambda} (\mu_\lambda(\alpha)),
\end{align*}
and hence
\[
 ({\omega_\lambda})_{\mu_\lambda(\alpha)} ( {\mu_\lambda}_* X_M, {\mu_\lambda}_* Y_M)
 = ({\omega_\lambda})_{\mu_\lambda(\alpha)} (X_{\Omega_\lambda}, Y_{\Omega_\lambda}).
\]
This implies the equation \eqref{e:mu_lambda is a moment map} 
by the definition \eqref{e:definition of omega_lambda} of $\omega_\lambda$.
\end{proof}

As an application,
one can obtain an explicit embedding of \Gr-orbit of $\lambda$ into the twisted cotangent bundle,
where $\Gr$ is a noncompact Hermitian Lie group.
In fact,
let $(\Gr,\Kr)$ be a Hermitian symmetric pair of noncompact type
whose complexifications are equal to $\Gc$ and $L$ respectively in the above.
Then if $\lambda$ satisfies the condition for which we referred to \cite{Knapp86} in Section 2, 
the holomorphic discrete series representation $(\pi_\lambda, \mathcal H_\lambda)$
possesses a unique element (highest weight vector) $\varphi_\lambda$ satisfying 
\begin{align*}
 \pi_\lambda(t) \varphi_\lambda &= \lambda(t) \varphi_\lambda 
   \quad (t \in T_{\R});
     \\
 \pi_\lambda(X) \varphi_\lambda &= 0 
   \qquad \quad (X \in \u); 
     \\
 \varphi_\lambda(e) &= 1,
\end{align*}
where $T_{\R}$ denotes a maximal torus of $\Kr$.
Note that $1/\varphi_{\lambda}$ is naturally extended to 
a holomorphic function defined on the dense open set $p^{-1}(U_e)$, 
which we still denote by the same notation.

Let us define a real-analytic function $f_{\lambda}: U_e \to \R$ by
\[
 f_\lambda(z) := \frac{1}{\varphi_{\lambda} (u_z^\dagger u_z)},
\]
where, for $g \in \Gc$, we write $g^\dagger = \tau(g)^{-1}$
with $\tau$ the involution that characterizes the real form $\Gr$ in \Gc.
Note that $f_\lambda$ is positive on $\Gr/\Kr$.
\begin{prop}
If one denotes by $\Omega_{\lambda}^{\,\R}$ the coadjoint orbit of $\lambda$ under $\Gr$,
then one finds that
\begin{equation}
\label{e:embed_real_orbit}
\mu_\lambda :  \left\{ -\dd' \log f_\lambda(z) ; z \in \Gr/\Kr \right\}  
   \overset{\sim}{\longrightarrow} \Omega_{\lambda}^{\,\R},
\end{equation}
where $\dd'$ denotes the holomorphic part of the exterior derivative $\dd$ .
\end{prop}
The proof is exactly the same as that of \cite[Proposition 3.3]{HOOS93}.

If we denote by $a^\R$ the canonical isomorphism 
from the real-analytic cotangent bundle $T^* (U_e)^\R$ onto the holomorphic cotangent bundle $T^*U_e$ 
given in terms of coordinates by 
\[
 \dd x^\alpha \longleftrightarrow \frac12 \dd z^\alpha
   \quad \textrm{and} \quad
 \dd y^\alpha \longleftrightarrow -\frac{\sqrt{-1}}{2} \dd z^\alpha
\]
with $z^\alpha=x^\alpha + \sqrt{-1} y^\alpha$ for each $\alpha \in \Delta(\u)$,
then the equation \eqref{e:embed_real_orbit} reads
\[
 \mu_\lambda^\R : \left\{ -\dd \, \log f_\lambda(z) ; z \in \Gr/\Kr \right\}  
   \overset{\sim}{\longrightarrow} \Omega_{\lambda}^{\,\R},
\]
where we put $\mu_\lambda^\R:=\mu_\lambda \circ a^\R$.
Thus, this is our version of \cite[Lemma 7.17]{SV98},
which plays a prominent role in establishing the character formula therein.


\bibliographystyle{amsplain}      
\bibliography{rep,geom}   

\end{document}